\documentclass[11pt, leqno]{amsart}
\usepackage{amsfonts,amssymb, amscd}
\usepackage{amsmath}
\usepackage{lmodern}
\usepackage{mathrsfs} 
\usepackage{amsthm}
\usepackage{qsymbols}
\usepackage{mathtools}
\mathtoolsset{showonlyrefs}
\usepackage{dsfont}
\usepackage{graphicx}
\usepackage{latexsym}
\usepackage{chngcntr}
\usepackage[noadjust]{cite}
\usepackage{paralist}
\usepackage[parfill]{parskip}
\usepackage{bm}
\usepackage{tikz-cd}
\usepackage{csquotes}

\usepackage[shortlabels]{enumitem}
\newtheorem{theorem}{Theorem}[section]
\newtheorem{lemma}[theorem]{Lemma}
\newtheorem{proposition}[theorem]{Proposition}

\theoremstyle{definition}
\newtheorem{definition}[theorem]{Definition}
\newtheorem{remark}[theorem]{Remark}
\newtheorem{example}[theorem]{Example}
\usepackage[plainpages=false,pdfpagelabels,
citecolor=red]{hyperref}
\usepackage{xcolor}
\hypersetup{
 colorlinks,
 linkcolor={cyan!90!black},
 citecolor={magenta},
 urlcolor={green!40!black}
} %important to load after amsthm and before the rest
\newcommand{\R}{\mathbb{R}}

\newcommand{\ExtZhou}{\boldsymbol{E}} % JW Extension
\newcommand{\ExtZero}{\mathrm{E}_0} % Zero Extension to fattened set
\newcommand{\Ext}{\mathrm{E}} % Extension operator for Wsp with Hardy term
\renewcommand{\L}{\mathrm{L}} 
\newcommand{\W}{\mathrm{W}} % intrinsic space -- closures for BC
\renewcommand{\d}{\,\mathrm{d}}

\newcommand{\B}{\mathrm{B}}

\newcommand{\cl}[1]{\overline{#1}}
\newcommand{\bO}{\boldsymbol{O}}

\DeclareMathOperator{\bd}{\partial \!}

\DeclareMathOperator{\dist}{d}
\DeclareMathOperator{\diam}{diam}

\def\XXint#1#2#3{{\setbox0=\hbox{$#1{#2#3}{%
\int}$ }
\vcenter{\hbox{$#2#3$ }}\kern-.6\wd0}}
\title[Extensions for fractional Sobolev spaces with boundary conditions]{The extension problem for fractional Sobolev spaces with a partial vanishing trace condition}
\author{Sebastian Bechtel}
\address{Fachbereich Mathematik, Technische Universit\"at Darmstadt, Schlossgartenstr. 7, 64289 Darmstadt, Germany}
\email{bechtel@mathematik.tu-darmstadt.de}
\date{\today}
\dedicatory{}
\begin{document}
\begin{abstract}

We construct whole-space extensions of functions in a fractional Sobolev space of order $s\in (0,1)$ and integrability $p\in (0,\infty)$ on an open set $O$ which vanish in a suitable sense on a portion $D$ of the boundary $\bd O$ of $O$. The set $O$ is supposed to satisfy the so-called \emph{interior thickness condition in $\bd O \setminus D$}, which is much weaker than the global interior thickness condition. The proof works by means of a reduction to the case $D=\emptyset$ using a geometric construction. 

\end{abstract}
\maketitle
%%%%%%%%%%%%%%%%%%%%%%%%%%%%%%%%%%%%%%%%%%%%%%%%%%%%%%%%%%%%%%%%%%%%%%%%%%%%%%%%%%%%%%%%%%%%%%%%%%%%%%%%%%%%%%%%%%%%%%%%%%%%%%%%%%%%%%%%%%%%%%%%%%%%%%%%%%%%%%%%%%%%
\section{Introduction and main results}
\label{Sec: Introduction}

Let $O\subseteq \R^d$ be open. For $s\in (0,1)$ and $p\in (0,\infty)$ the fractional Sobolev space $\W^{s,p}(O)$ consists of those $f\in \L^p(O)$ for which the seminorm
\begin{align}
	[f]_{\W^{s,p}(O)} \coloneqq \left( \iint_{\substack{x,y \in O \\ |x-y| < 1}} \frac{|f(x)-f(y)|^p}{|x-y|^{sp+d}} \d y \d x \right)^\frac{1}{p}
\end{align}
is finite. Under the interior thickness condition
\begin{align}
\label{ITC}
	\forall x\in O, r\in (0,1] \, : \quad |\B(x,r) \cap O| \gtrsim |\B(x,r)|, \tag{ITC}
\end{align}
whole-space extensions for $\W^{s,p}(O)$ were constructed by Zhou~\cite{Zhou}. Though the mapping is in general not linear, extensions depend boundedly on the data. The case $p\geq 1$ was already treated earlier by Jonsson and Wallin~\cite{JW}, and their extension operator is moreover linear. In fact, Zhou has shown that the interior thickness condition is equivalent for $\W^{s,p}(O)$ to admit whole-space extensions. If we impose a vanishing trace condition on $\bd O$ in a suitable sense, zero extension is possible, so in this case no geometric quality of $O$ is needed. It is now natural to ask what happens if a vanishing trace condition is only imposed on a portion $D\subseteq \bd O$. 

To be more precise, we consider the space $\W^{s,p}_D(O)$ given by $\W^{s,p}(O) \cap \L^p(O, \dist_D^{-sp})$, where $\dist_D$ is the distance function to $D$. The fractional Hardy term in-there models the vanishing trace condition on $D$, compare for~\cite{Dyda-Vahakangas,Hardy-Poincare,ET,Hajlasz-PointwiseHardy,Lehrback-PointwiseHardy}. Spaces of this kind were also recently investigated in~\cite{KondratievProperties} and have a history of successful application in the theory of elliptic regularity, see for example~\cite{Hansen}.

The present paper seeks minimal geometric requirements under which functions in $\W^{s,p}_D(O)$ can be boundedly extended to whole-space functions. We will see in Lemma~\ref{Lem: ITC in boundary suffices} that in~\eqref{ITC} we could equivalently consider balls centered in $\bd O$ instead of $O$. Put $N \coloneqq \bd O \setminus D$. In Definition~\ref{Def: ITC in F} we introduce the \emph{interior thickness condition in $N$}, which requires that~\eqref{ITC} holds for balls centered in $N$. For $D=\emptyset$, this is just the usual interior thickness condition in virtue of the aforementioned Lemma~\ref{Lem: ITC in boundary suffices}. It is the main result of this article to show that the interior thickness condition in $N$ is sufficient for the $\W^{s,p}_D(O)$-extension problem.

A major obstacle is that the interior thickness condition in $N$ does not provide thickness in \emph{any} neighborhood around $N$, which makes localization techniques not applicable. An example for this is a self-touching cusp, see Example~\ref{Ex: geometry for extension operator}. Our construction is as follows. The extension procedure decomposes into a zero extension from $O$ to some suitable superset $\bO$ of $O$, which is an enlargement of $O$ near $D$, followed by an application of Zhou's construction on $\bO$. Hence, suitability of $\bO$ is measured by two properties: First, the zero extension can be bounded in $\W^{s,p}(\bO)$ with the aid of the fractional Hardy term. Second, $\bO$ satisfies~\eqref{ITC}, so that Zhou's result is applicable. A similar construction of $\bO$ was performed by the author together with M.~Egert and R.~Haller-Dintelmann in~\cite{Kato}. The main result then reads as follows.

\begin{theorem}
\label{Thm: Extension operator}
	Let $O\subseteq \R^d$ and let $D\subseteq \bd O$, $p\in (0,\infty)$ and $s\in (0,1)$. If $O$ satisfies the interior thickness condition in $\bd O\setminus D$, then there exists a bounded mapping
	\begin{align}
		\Ext: \W^{s,p}(O) \cap \L^p(O, \dist_D^{-sp}) \to \W^{s,p}_D(\R^d).
	\end{align}
	If $p\geq 1$, then $\Ext$ is moreover linear.
\end{theorem}

We will also comment on the sharpness of our result in Section~\ref{Sec: necessity}.

Finally, a remark on the case $p=\infty$ is in order. In this situation, the fractional Sobolev space is substituted by the H\"older space of order $s\in (0,1)$. Then the Whitney extension theorem~\cite[Thm.~3, p.174]{Stein} provides a linear extension operator without any geometric requirements. In particular, the fractional Hardy term is not needed, though it is easily seen that $\| f \dist_D^{-s} \|_\infty$ can only be finite if $f$ vanishes identically on $D$, and the same is of course true for the extension.

%To conclude, we consider the necessity of the geometric assumption from Section~\ref{Sec: Extension Operator Wsp} in Section~\ref{Sec: necessity}. To be more precise, we introduce a condition in Definition~\ref{Def: degenerate ITC} that is strictly weaker than that from Definition~\ref{Def: ITC in F} imposed in Theorem~\ref{Thm: Extension operator}. Proposition~\ref{Prop: degenerate condition necessary} shows that this condition is necessary for extension operators on the space $\W^{s,p}(O)\cap \L^p(O, \dist_D^{-sp})$. Example~\ref{Ex: necessary condition attained} is a geometry in which such an extension operator is available but which is not admissible in Theorem~\ref{Thm: Extension operator}. 

\subsection*{Acknowledgments}

The author thanks his Ph.D.\ advisor Robert Haller-Dintelmann for his support, the \enquote{Studienstiftung des deutschen Volkes} for financial and academic support, Joachim Rehberg for suggesting the topic and Juha Lehrb\"ack for valuable discussions.

\subsection*{(Non-)Standard notation}

We write $\B(x,r)$ for the open ball around $x$ with radius $r$. The closure of a set $A$ is denoted by $\cl{A}$ and the Lebesgue measure of $A$ is denoted by $|A|$. If we integrate with respect to the Lebesgue measure, we write $\d x$, $\d y$ and so on. For diameter and distance induced by the Euclidean metric we write $\diam(\cdot)$ and $\dist(\cdot,\cdot)$. Also, the shorthand notation $\dist_E(x)\coloneqq \dist(\{x\}, E)$ is used. We employ the notation $\lesssim$ and $\gtrsim$ for estimates up to an implicit constant that does not depend on the quantified objects. If two quantities satisfy both $\lesssim$ and $\gtrsim$ we write $\approx$.

\section{Geometry}
\label{Sec: Geometry}

\begin{definition}
\label{Def: ITC in F}
	Let $E\subseteq \R^d$ and $F\subseteq \bd E$. Then $E$ satisfies the \emph{interior thickness condition in $F$} if
	\begin{align}
		\forall x\in F, r\in (0,1] \, : \quad |\B(x,r) \cap E| \gtrsim |\B(x,r)|.
	\end{align}
\end{definition}
The following lemma shows the equivalence between the~\eqref{ITC} condition with balls centered in $O$ and~\eqref{ITC} with balls centered in $\bd O$ already mentioned in the introduction. Though its proof is simple, we include it for good measure.
\begin{lemma}
\label{Lem: ITC in boundary suffices}
	Let $E\subseteq \R^d$. Then $E$ satisfies~\eqref{ITC} if and only if $E$ satisfies the interior thickness condition in $\bd E$.
\end{lemma}
\begin{proof}
	Assume~\eqref{ITC} and let $x\in \bd E$, $r\in (0,1]$. Then pick some $y\in \B(x,r/2)\cap E$ and calculate
	\begin{align}
		|\B(x,r)\cap E| \geq |\B(y,r/2)\cap E| \gtrsim |\B(y,r/2)| \approx |\B(x,r)|. 
	\end{align}
	Conversely, let $x\in E$, $r\in (0,1]$ and assume that $E$ is interior thick in $\bd E$. If $\B(x,r/2) \subseteq E$ then the claim follows immediately. Otherwise, pick again some $y \in \B(x,r/2)\cap \bd E$ and argue as above.
\end{proof}
The following simple example shows that a set can satisfy the thickness condition in some closed subset of the boundary but fails to have it in any neighborhood of it.
\begin{example}
\label{Ex: geometry for extension operator}
	Consider $O=\{ (x,y)\in \R^2\colon |y| < x^2, x<0 \} \cup \{ (x,y)\in \R^2\colon x>0 \}$. This means that $O$ consists of the right half-plane touched by a cusp from the left. Put $D$ to be the boundary of the cusp and $N$ is the $y$-axis except the origin. Then the~\eqref{ITC} estimate holds in $N$ since each ball centered in $N$ hits the half-plane with half of its area, but any proper neighborhood around $N$ would contain a region around the tip of the cusp, in which thickness does not hold (consider a sequence that approximates the tip of the cusp and test with balls that do not reach $N$).
\end{example}

\section{The extension operator}
\label{Sec: Extension Operator Wsp}

In this section we prove Theorem~\ref{Thm: Extension operator}. First, we construct $\bO$ and show that it is interior thick. Second, we show that the zero extension to $\bO$ is bounded using a simple geometric argument. Finally, we patch everything together to conclude. Throughout, $O$ and $D$ are as in Theorem~\ref{Thm: Extension operator} and we put $N\coloneqq \bd O \setminus D$ for convenience.

\subsection{Embedding into an interior thick set}
\label{Sec: fattening}

We construct an open set $\bO \subseteq \R^d$ with $O\subseteq \bO$, $\bd O \subseteq \bd \bO$ and that satisfies~\eqref{ITC}. According to the assumption on $N$ and Lemma~\ref{Lem: ITC in boundary suffices} it suffices to check that $\bO$ is interior thick in $D$ and the \enquote{added} boundary. Of course we could take $\bO$ as $\R^d \setminus \bd O$ in this step but this would make zero extension in Section~\ref{Sec: zero extension} impossible. Therefore, our construction will be in such a way that moreover $|x-y| \gtrsim \dist_D(x)$ whenever $x\in O$ and $y\in \bO \setminus O$, see Lemma~\ref{Lem: point distance for zero extension}, which will do the trick in step two.

Let $\{Q_j\}_j$ be a Whitney decomposition for the complement of $\overline{N}$, which means that the $Q_j$ are disjoint dyadic open cubes such that
\begin{align}
	\mathrm{(i)}\quad \bigcup_j \overline{Q_j} = \R^d \setminus \overline{N}  \qquad \mathrm{(ii)}\quad \diam(Q_j) \leq \dist(Q_j, N) \leq 4 \diam(Q_j).
\end{align}
Using the Whitney decomposition we define
\begin{align}
	\Sigma\coloneqq \{ Q_j\colon \cl{Q_j} \cap \cl{O} \neq \emptyset \} \qquad\text{and}\qquad \bO \coloneqq O \cup \Bigl( \bigcup_{Q\in \Sigma} Q\setminus D \Bigr).
\end{align}
Note that for $Q\in \Sigma$ one has $Q\setminus D = Q\setminus \bd O$. Then all claimed properties of $\bO$ except~\eqref{ITC} follow immediately by definition.
So, let $x\in \bd \bO$ and $r\in (0,1]$. If $x \in \cl{N}$ then we are done by assumption (keep Lemma~\ref{Lem: ITC in boundary suffices} in mind). Otherwise, either $x\in D$ or $x\in \bd Q$ for some $Q\in \Sigma$ (to see this, use that the Whitney decomposition is locally finite). But if $x\in D$ then $x\in \overline{Q}$ for some $Q\in \Sigma$ by property (i) of the Whitney decomposition and the definition of $\Sigma$. Hence, in either case $x\in \overline{Q}$ for some $Q\in \Sigma$. Now we make a case distinction on the radius size compared to the size of $Q$. If $r \geq 4 \dist(Q, N)$, pick $y\in \overline{Q}$ and $z\in \overline{N}$ with $\dist(Q,N) = |y-z|$. Then with (ii) we get
\begin{align}
	|x-z| \leq |x-y| + |y-z| \leq \diam(Q) + \dist(Q,N) \leq 2 \dist(Q,N) \leq r/2,
\end{align}
hence $\B(x,r)$ contains a ball of radius $r/2$ centered in $\overline{N}$ and we are done. Otherwise, if $r<4\dist(Q,N)$, then by (ii) we get $r<16\diam(Q)$ and the claim follows from~\eqref{ITC} for $Q$.

\subsection{Zero extension}
\label{Sec: zero extension}

Let $\bO$ denote the set constructed in the previous step. We define the zero extension Operator $\ExtZero$ from $O$ to $\bO\cup D$ and claim that it is $\W^{s,p}(O) \cap \L^p(O, \dist_D^{-sp}) \to \W^{s,p}(\bO)$ bounded. We start with a preparatory lemma.
\begin{lemma}
\label{Lem: point distance for zero extension}
	One has $|x-y| \gtrsim \dist_D(x)$ whenever $x\in O$ and $y\in \bO\setminus O$.
\end{lemma}
\begin{proof}
	We consider $y\in \bO\setminus O$ and pick some $Q\in \Sigma$ that contains $y$. We distinguish whether or not $x$ and $y$ are far away from each other in relation to $\diam(Q)$.
	
	\emph{Case 1}: $|x-y|<\diam(Q)$. Fix a point $z \in \bd O$ on the line segment connecting $x$ with $y$. Assume for the sake of contradiction that $z\in N$. Then using (ii) we calculate
	\begin{align}
		\dist(Q,N) \leq |y-z| \leq |x-y| < \diam(Q) \leq \dist(Q,N),
	\end{align}
	hence we must have $z\in D$. Thus, $|x-y| \geq |x-z| \geq \dist_D(x)$.
	
	\emph{Case 2}: $|x-y| \geq \diam(Q)$. By definition of $\Sigma$ and $y\not\in O$ we can pick $z\in \cl{Q}\cap D$. Then
	\begin{align}
		|x-z| \leq |x-y| + |y-z| \leq |x-y| + \diam(Q) \leq 2 |x-y|,
	\end{align}
	hence $|x-y| \gtrsim \dist_D(x)$.
\end{proof}
This enables us to estimate $\ExtZero$. Clearly, we only have to estimate the $\W^{s,p}(O)$--seminorm since extension by zero is always isometric on $\L^p$. Let $f \in \W^{s,p}(O) \cap \L^p(O,\dist_D^{-sp})$, then
\begin{align}
\begin{split}
\label{Eq: E0 estimate}
	\iint_{\substack{x,y \in \bO \\ |x-y| < 1}} \frac{|\ExtZero f(x)-\ExtZero f(y)|^p}{|x-y|^{sp+d}} \d y \d x &\leq \iint_{\substack{x,y \in O \\ |x-y| < 1}} \frac{|f(x)-f(y)|^p}{|x-y|^{sp+d}} \d x \d y \\ 
	&+ 2 \iint_{\substack{x\in O, y\in (\bO\setminus O) \\ |x-y| < 1}} \frac{|f(x)|^p}{|x-y|^{sp+d}} \d x \d y.
\end{split}
\end{align}
The first term is bounded by $\|f\|_{W^{s,p}(O)}^p$, so it only remains to bound the second term. Using Lemma~\ref{Lem: point distance for zero extension} and calculating in polar coordinates we find
\begin{align}
	\int_{\substack{y \in (\bO\setminus O) \\ |x-y| < 1}} |x-y|^{-sp-d} \d y \lesssim \dist_D(x)^{-sp}.
\end{align}
Plugging this back into~\eqref{Eq: E0 estimate} yields that we can bound the second term therein by the Hardy term $\|f\|_{\L^p(O, \dist_D^{-sp})}^p$.

\subsection{Proof of Theorem~\ref{Thm: Extension operator}}
\label{Sec: Proof of Wsp Ext Theorem}

We combine the results from the previous sections with the extension procedure of Zhou to conclude.
\begin{proof}[Proof of {Theorem~\ref{Thm: Extension operator}}]
	Put $\Ext = \ExtZhou \circ \ExtZero$, where $\ExtZhou$ is the (non-linear) extension operator of Zhou and $\ExtZero$ is the zero extension operator from the previous step. Clearly, $\ExtZero$ is linear, and we have seen in Section~\ref{Sec: zero extension} that it is $\W^{s,p}_D(O) \to \W^{s,p}(\bO)$ bounded. Since $\bO$ satisfies~\eqref{ITC} by Section~\ref{Sec: fattening}, $\ExtZhou$ is well-defined on $\W^{s,p}(\bO)$ and bounded into $\W^{s,p}(\R^d)$ by Zhou's result.
	
	The claim for $p<1$ then follows already by composition. In the case $p\geq 1$, note that $\ExtZhou$ can be constructed to be linear, see also~\cite{JW}.
\end{proof}
\section{On the sharpness of our result}
\label{Sec: necessity}
In this final section we take a look on how close to a characterization our condition is. We will see in Example~\ref{Ex: necessity} that the interior thickness condition in $N$ is not necessary for the extension problem, but that our construction might fail without it. Afterwards, we will introduce a \emph{degenerate interior thickness condition in $N$}, which is necessary for the extension problem, but is not sufficient for our construction.

\begin{example}
\label{Ex: necessity}
	Consider the upper half-plane in $\R^2$. A Whitney decomposition can be constructed from layers of dyadic cubes. Let $O$ be a \enquote{cusp} that is build from those Whitney cubes which intersect the area below the graph of the exponential function, and let $N$ be its lower boundary given by the real line in $\R^2$. It is eminent that $O$ is not interior thick in $N$. Moreover, our construction of $\bO$ just adds another layer of cubes, so $\bO$ is of the same geometric quality. Hence, our construction does not work in this situation. But zero extension to the upper half-plane is still possible, so with $\bO$ chosen as the upper half-plane, we can construct an extension procedure for $\W^{s,p}_D(O)$ in this configuration. This shows that the interior thickness condition in $N$ is not necessary for $\W^{s,p}_D(O)$ to admit whole-space extensions, but is \enquote{necessary} for our construction to work.
\end{example}

We introduce the aforementioned modified version of the interior thickness condition in $N\subseteq \bd O$ that degenerates near $\bd O \setminus N$.
\begin{definition}
\label{Def: degenerate ITC}
	Say that $O$ satisfies the \emph{degenerate interior thickness condition in $N$} if $O\subseteq \R^d$ is open, $N \subseteq \bd O$ and they fulfill
	\begin{align}
		\forall x\in N, r\leq \min(1, \dist_{\bd O \setminus N}(x))\colon  |\B(x,r)\cap O| \gtrsim |\B(x,r)|.
	\end{align}
\end{definition}
In fact, this condition is necessary for the $\W^{s,p}(O)\cap \L^p(O, \dist_D^{-sp})$-extension problem. The technique to show this is due to Zhou~\cite{Zhou}. By the restriction in radii, the test functions used in Zhou's proof belong to $\W^{s,p}_D(O)$, and then his proof applies \emph{verbatim}, hence we omit the details.
\begin{proposition}
\label{Prop: degenerate condition necessary}
	Let $O \subseteq \R^d$ be open, $D\subseteq \bd O$, $p\in (0,\infty)$, $s\in (0,1)$ and put $N\coloneqq \bd O \setminus D$. If $\W^{s,p}_D(O)$ admits whole-space extensions, then $O$ satisfies the degenerate interior thickness condition in $N$.
\end{proposition}

\begin{remark}
\label{Rem: degenerate ITC not sufficient}
	In Example~\ref{Ex: necessity} we have seen a configuration which admits whole-space extension for $\W^{s,p}_D(O)$-functions, so in this situation, $O$ satisfies the degenerate interior thickness condition near $N$ by Proposition~\ref{Prop: degenerate condition necessary} (of course, this can also be seen directly). On the other hand, we have seen in that example that in this configuration our construction does not work. Hence, the degenerate interior thickness condition in $N$ is too weak for our proof of Theorem~\ref{Thm: Extension operator}.
\end{remark}

%%%%%%%%%%%%%%%%%%%%%%%%%%%%%%%%%%%%%%%%%%%%%%%%%%%%%%%%%%%%%%%%%%%%%%%%%%%%%%%%%%%%%%%%%%%%%%%%%%%%%%%%%%%%%%%%%%%%%%%%%%%%%%%%%%%%%%%%%%%%%%%%%%%%%%%%%%%%%%%%%%%%

\end{document}